\documentclass[12pt]{article}
\usepackage{amsmath, amsthm, amscd, amsfonts, amssymb, graphicx, color}
\usepackage[bookmarksnumbered, colorlinks, plainpages]{hyperref}
\newtheorem{theorem}{Theorem}[section]

\newtheorem{proposition}[theorem]{Proposition}

\theoremstyle{definition}
\newtheorem{definition}[theorem]{Definition}
\newtheorem{example}[theorem]{Example}

\theoremstyle{remark}
\newtheorem{remark}{Remark}[section]
\numberwithin{equation}{section}

\begin{document}

\setcounter{page}{1}


\begin{center}
{\Large \textbf{$X$-convexity and Applications of Quasi-$X$-Convex Functions }}

\bigskip

\textbf{Musavvir Ali$^{a,*}$} and \textbf{Ehtesham Akhter$^b$}\\
\textbf{}  \textbf{}\\

{\small $^{a,~b}$  Department of Mathematics,\\ Aligarh Muslim University, Aligarh-202002, India}
\end{center}
\noindent
\footnote{$^*$Corresponding author\\
E-mail addresses: musavvir.alig@gmail.com (M. Ali),\\ ehteshamakhter111@gmail.com (E. Akhter). }
\bigskip

{\abstract
A class of real functions, which is the generalization of a family of convex functions, is introduced; in this connection, we have defined $X$-convex, strictly $X$-convex, quasi-$X$-convex, strictly quasi-$X$-convex, and semi-strictly quasi-$X$-convex functions. Moreover, in this paper, we give a detailed study of the fundamental properties of these functions with various examples, supporting the concepts. Finally, the study of optimization problems employs quasi-$X$-convex, semistrictly quasi-$X$-convex, and strictly quasi-$X$-convex functions.}\\

{\noindent  \bf Keywords:}  Convexity, Inequality, $X$-convex function, Quasi $X$-convex function, Monotonicity, Optimization.

\section{Introduction} Due to the vital role of convex functions in different fields of mathematics, the study of convex functions is a favourite topic for researchers. Not only in hardcore mathematics but also in the fields of Economics and Engineering, ``convexity" has many great applications. The concept of convexity on sets and functions has been a burning topic and is one of the most explored topics of research from the last many decades due to its compatibility with optimization techniques and applications in minimization problems of convexity. More than a century ago, Jensen \cite{Jensen} introduced the convex functions and then in 1949, Fenchel \cite{Fenchel  On conjugate convex functions},\cite{Fenchel Convex Cones Sets and Functions} gave a thorough study of these functions and related topics. The primary purpose of this paper is to provide a new type of generalization of convexity and to introduce members in a huge family of real functions, viz; $X$-convex, strictly $X$-convex, quasi-$X$-convex, strictly quasi-$X$-convex, semi-strictly quasi-$X$-convex functions.

On a convex subset of $R^n$, a real function $\phi$ is called a convex function, if for each $x_1, x_2 $ in the domain, the following equation holds good
\begin{equation}
\phi(q_1x_1 + q_2x_2) \leq q_1\phi(x_1) + q_2\phi(x_2),
\end{equation}
where $q_1, q_2$ are non-negative real numbers such that $q_1 + q_2 = 1$.\\

Equation (1.1) can also be explained as an arithmetic weighted mean value under a convex function $\phi$ associated with any two points of the domain which is less than or equal to the weighted mean of different values of the points under $\phi$, provided the same weights are used. 

Equation (1.1) leads to its generalization as $X$-convexity by replacing the terms on the right with a more general mean at $x_1$ and $x_2$ under the function $\phi$. Therefore, we have generated a more prominent family of functions have been generalized than the simple convex functions, and the desirable properties of convex functions have also been retained.

In the literature of convex functions, we have read many generalizations useful for the optimization (c.f., \cite{Crouzrix asiconvex functions},\cite{Greenberg}, and \cite{Mangarasian Pseudo  Convex Function}).

Recently,  Kili{\c{c}}man \cite{Almutairi} introduced the concept of $(h-m)$-convexity. Alizadeh \cite{Alizadeh} introduced the concept of $e$-convexity. Sial \cite{Sial} introduced the concept of $(\alpha,m)$-convexity and several new generalization of convexity can be seen in (\cite{Abdulaleem}, \cite{Cambini}, \cite{Du}, \cite{Kashuri}, \cite{Li}, \cite{Mahreen}, \cite{Shi}, \cite{Trean}, \cite{Yang}) and the reference cited therein.

Inspired by prior research endeavours and the significance of the notions of convexity and generalized convexity in this paper, we will be discussing the relationship between $X$-convexity with other generalizations on convexity, already established. Section 2 deal will the preliminaries on convexity and related topics. In section 3, $X$-convex, quasi-$X$-convex, strictly quasi-$X$-convex, semi-strictly quasi-$X$-convex functions will be discussed, and it is shown that there is quite a natural process to proceed for quasi-convexity, strictly quasi-convexity, semi-strictly quasi-convexity of the real functions. We have also demonstrated desirable results of $X$-convex functions, and it is derived that $X$-convex functions can be reduced to ordinary convex functions along with a simple transformation.

 \section{Preliminaries}
	
	\begin{definition}
	 Suppose that $ \phi:R^n$ $\rightarrow$ $R\cup{\{\pm}\infty\}$   is a function, an extended and real-valued. Then\\
	
	(a) $\{(r, \eta) \in R^n \times R : \phi(r) \leq \eta \}$ is a set, called as	epigraph of function $\phi$, and it is denoted by $epi(\phi)$.\\
	
	(b)  $\{(r, \eta) \in R^n \times R : \phi(r) \geq \eta \}$ is a set, called as	hypo-graph of function $\phi$, and it is denoted by $hyp(\phi)$.\\
	 
	(c)	$\{r \in R^n : \phi(r) \leq \eta \}$ is a set, called as	lower level set on function $\phi$, and it is  denoted by $L(\phi,\eta)$.\\
	
	(d) $\{r \in R^n : \phi(r) \geq \eta \}$ is a set, called as	upper level set on function $\phi$, and it is  denoted by $U(\phi,\eta)$.\\
	
\end{definition}
		
		For any two distinct points $x_1$ and $x_2$  $\in R^n$, the set $ \{t \in R^n : t=\delta x_1 + (1- \delta)x_2 ~~~  $for  all$ ~~   \delta \in   R\} $ is defined to be a line through $x_1$ and $x_2$, while the set $ [x_1 ,x_2] = \{ t\in R^n : t= \delta x_1+ (1-\delta)x_2 ~~  for  ~~ 0 \leq \delta  \leq  1\} $ is the  line segments through $x_1$ and $x_2$.
		
		\begin{definition}
			$R^n$\textsc{\char13}s non-empty convex subset $M$ is known as\\
			
		(i) {\bf convex set } $\forall$ points $ r,t \in M $ and $ \eta, \mu \geq 0 ~ such~ that~ \eta + \mu =1, $ if we have $ \eta r+\mu t \in M, $ or simply we write that $\forall$ $ r,t \in M $ and $\delta \in [0,1],$ we get $\delta r+(1-\delta) t \in M. $	
		\end{definition}
\begin{definition}
On a convex subset (not empty) $M$ of $R^n$, a function $ \phi:M \rightarrow R$ is known as\\
	
	(i) {\bf convex} if $\forall$ points  $r$ and $t \in M $, and  $ 0 \leq \delta \leq 1$, such that  $$ \phi(\delta r + (1- \delta) t) \leq \delta \phi(r) + (1- \delta) \phi(t);$$
	
   (ii) {\bf  strictly convex} if $\forall$ points  $r$ and $t \in M $, $r \neq t$ and  $ 0 < \delta < 1$, such that  $$ \phi(\delta r + (1- \delta) t) < \delta \phi(r) + (1- \delta) \phi(t);$$

 If -$\phi$ is (strictly) convex, a function $ \phi:M \rightarrow R$ is said to be (strictly) concave.
\end{definition}
	\begin{definition}
Let $ \phi:M \rightarrow R$ be defined on a convex, non-empty subset $M$ of $R^n$,  then  $ \phi$ is called as\\
	
		(i) {\bf quasi-convex} if $\forall$ points  $r$ and $t \in M $, and  $ 0 \leq \delta \leq 1$, such that  $$ \phi(\delta r + (1- \delta) t) \leq \mbox{max} \{ \phi(r),~ \phi(t)\};$$

		(ii) {\bf  strictly quasi-convex} if $\forall$ points  $r$ and $t \in M $, $r \neq t$ and  $ 0 < \delta < 1$, such that  $$ \phi(\delta r + (1- \delta) t) <   \mbox{max} \{ \phi(r),~ \phi(t)\};$$

		(iii) {\bf semi-strictly quasi-convex} if $\forall$ points  $r$ and $t \in M $, $\phi(r) \neq \phi(t)$ and  all $ 0 < \delta < 1,$    $$  \phi(\delta r + (1- \delta) t) < \mbox{max} \{\phi(r),\phi(t)\}, $$

		If $-\phi$ is (strictly, semi-strictly) quasi-convex, a function $ \phi : M  \rightarrow R $ is said to be (strictly, semi-strictly) quasi-concave.	
		\end{definition}  
		\section{$X$-convexity}	
		
				The notion of convexity  is intensely studied and, its various generalizations have been explored by many researchers (\cite{Bertsekas Convex Analysis}-\cite{Rockafellar}). In this section, we have given some new generalizations of the convexity of sets and functions. Different definitions are analyzed, and the concepts are verified through the examples.
		\begin{definition} 
	An $X$-convex set is a non-empty subset $M$ of  $R^n$, if $\forall$ elements $r,~t$ of set  $M $ and $ 0 \leq \delta \leq 1,$ $\exists$ a vector valued map $ g : M  \rightarrow R^n $    such that,  we get   $\delta (r-t)+g(t)  \in M.$ 
          
			\end{definition}
		\begin{remark}

		 In general, every convex set is also an $X$-convex set, however converse isn\textsc{\char13}t always true.
		\begin{example}
			
			The set $ [1,2] \cup [3, \infty) $ is not convex set but with map $ g :  [1,2] \cup [3, \infty) \rightarrow R $, given by $ g(t)= t+3,$ it is an $X$-convex. 
		\end{example}
	\end{remark}
\begin{definition}
$ \phi : M \rightarrow   R$, defined on an $X$ - convex, non-empty set $M \subseteq R^n$, is

(a) {\bf $X$-convex} if $\forall$ $ r ,t \in M$ and  $0 \leq \delta \leq 1, $  $\exists$ a vector valued map $ g : M \rightarrow R^n $ such that,  $$ \phi(\delta (r-t) + g(t)) \leq \delta \phi(r) + (1- \delta) \phi(t) , $$

(b) {\bf strictly $X$-convex} if $\forall$ $ r ,t \in M$, $r \neq t$ and  $0<\delta < 1, $  $\exists$ a vector valued map  $ g : M \rightarrow R^n $ such that,     $$ \phi(\delta (r-t) + g(t)) < \delta \phi(t) + (1- \delta) \phi(t), $$

(c) $ \phi : M  \rightarrow R $ is (strictly) an $X$-concave function in case $-\phi$ is (strictly) $X$-convex.

		\begin{remark}
		Every convex function is an $X$-convex function by definition (2.3(a)), however the converse may not be true.
	 
		\end{remark}	
		\begin{example}
We take a constant function $ \phi :  [1,2] \cup [3, \infty)  \rightarrow R$ defined as  $\phi(r)=c$ on a set  $[1,2] \cup [3, \infty)$,  where c is a constant, and let the function $ g :  [1,2] \cup [3, \infty)  \rightarrow R $ be defined as: $$ g(t)=t+3. $$
	
	Then, $\phi$ is not a convex function as its domain is not a convex set. But $\phi$ is an $X$-convex function with respect to map $g$.
		\end{example}
	
	\begin{example}
	
	We take an identity function $ \phi$ on $R$ and let a map $ g : R \rightarrow R $ be defined as: $$g(t)=t-\alpha,$$ where $\alpha (> 0) $, a constant.\\The function $\phi $ is strictly $X$-convex with respect to map $g$. 
	\end{example}	
		\end{definition}
	
	\begin{definition}
	$ \phi : M \rightarrow   R$, defined on an $X$ - convex, non-empty set $M \subseteq R^n$, is
						
				(a) {\bf quasi-$X$-convex} if $\forall$ $ r ,t \in M$ and  $0 \leq \delta \leq 1, $  $\exists$ a vector valued map $ g : M \rightarrow R^n $ such that,
				  $$ \phi(\delta (r-t) + g(t)) \leq \mbox{max} \{\phi(r),\phi(t)\}, $$ 
				
				(b) {\bf strictly quasi-$X$-convex} if $\forall$ $ r ,t \in M$, $r\neq t $ and  $0 < \delta < 1, $  $\exists$ a vector valued map $ g : M \rightarrow R^n $ such that,    $$ \phi(\delta (r-t) + g(t)) < \mbox{max} \{\phi(r),\phi(t)\}, $$
				
				(c) {\bf semi-strictly quasi-$X$-convex} if $\forall$  $ r ,t \in M$ and  $0 < \delta < 1, $  $\exists$ a vector valued map $ g : M \rightarrow R^n $ such that, $ \phi(r) \neq \phi(t) $     $$\phi((\delta (r-t) + g(t)) < \mbox{max} \{\phi(r),\phi(t)\}$$

				If $-\phi$ is (strictly, semi-strictly) quasi-$X$-convex, $ \phi : M  \rightarrow R $ is said to be (strictly, semi-strictly) quasi-$X$-concave.
			\end{definition} 
		
		\begin{example}
		
			Suppose $ \phi :  (-\infty, -3 ] \cup [-2, -1]  \rightarrow R $ be given by: $$ \phi(r)=\alpha+[r],$$ \\where $ [ , ] $ is the greatest integers function, $\alpha$ is any constant number and let the function $ g :  (-\infty, -3 ] \cup [-2, -1]  \rightarrow R $ be given by: $$ g(t)=t- 3. $$ Then $\phi$ is a quasi-$X$-convex with respect to map $g$.
		\end{example} 
		\begin{remark}
		It is evident that the family of (strictly) $X$-convex functions implies the family of (strictly) quasi-$X$-convex but the converse need not be true in general.
		\end{remark}
	
	\begin{example}
	
	Let $ \phi :  (-\infty, \frac{-1}{50} ] \cup [\frac{-1}{100}, 0]  \rightarrow R $ be defined as: $$ f(r)=\alpha+[r],$$ where $ [ , ] $ denotes the greatest integer function, $\alpha $ is any constant number and let $ g :  (-\infty, \frac{-1}{50} ] \cup [\frac{-1}{100}, 0]  \rightarrow R $ be a real valued function defined by: $$ g(t)=t- \frac{1}{50}. $$ The function $\phi$ is quasi-$X$-convex with respect to $g$. However, it is not an $X$-convex function with respect to map $g$ as the choice $ \delta=0.502, r=-1.5~ \mbox{and}~ t=-2.5 $, yields $$ \phi(\delta (r-t) + g(t)) > \delta \phi(r) + (1- \delta) \phi(t). $$
	\end{example}
\begin{example}
	Suppose a function $\phi :  [-1, \frac{-1}{2} ] \cup [0, \infty)  \rightarrow R $ is given by: $$ \phi(r)= \begin{cases}

3, \text{ if } r=0 \\
2, \text{ if } r\neq 0 
	\end{cases},$$  \\ and let the function $ g :  [-1, \frac{-1}{2} ] \cup [0, \infty)  \rightarrow R $ be given by: $$ g(t)=t+ 1. $$

So, $\phi$ is semi-strictly quasi-$X$-convex with respect to map $g$, but neither strictly quasi-$X$-convex nor quasi-$X$-convex.

\end{example}
	\begin{remark}
An $X$-convex function is semi-strictly quasi-$X$-convex function, but the converse may not generally hold.
	\end{remark}

\begin{example}
	Let $ \phi :  [0, 2 ] \cup [5, \infty)  \rightarrow R $ be a function given by: $$ \phi(r)= \begin{cases}
	
	2, \text{ if } r=0 \\
	1, \text{ if } r\neq 0 
\end{cases},$$  \\ and let  $ g :  [0, 2 ] \cup [5, \infty)  \rightarrow R $ be a function defined as: $$ g(t)=t+ 5. $$ So, $\phi$ is semi-strictly quasi-$X$-convex but neither strictly quasi-$X$-convex nor $X$-convex function. 

\end{example}
\begin{remark}

It is evident from the definition that a strictly quasi-$X$-convex function is a quasi-$X$-convex function. However, the converse may not generally hold.
\end{remark}
\section{Some Properties of $X$-convex Functions}
\begin{proposition}

The epigraph of the $X$-convex function $ \phi : M \rightarrow R $ is an $X$-convex set, where $M$ is any non-empty $X$-convex subset  of $R^n$.  
\begin{proof}
	Consider $\phi$ be any $X$-convex function defined on a non-empty subset $M$ of $R^n$, and let $(r, \eta), (t, \mu) \in epi(\phi), $ then
	\begin{equation}
	 \phi(r) \leq \eta  \text{ and } \phi(t) \leq \mu.	
	\end{equation} 
	since $\phi$ is an $X$-convex function, then  $\forall$ $r,t \in M,$  $0 \leq \delta \leq 1,$  $\exists$ $ g : M \rightarrow R^n $, a vector valued map, with 
	$$ \phi(\delta (r-t) + g(t)) \leq \delta \phi(r) + (1- \delta) \phi(t) , $$
	$$\phi(\delta (r-t) + g(t)) \leq \delta \eta + (1-\delta) \mu , $$
	$$ (\delta (r-t) + g(t)) ,~\delta \eta + (1- \delta) \mu) \in epi(\phi),$$
	or	$$ (\delta (r-t) + g(t)) ,~ \alpha) \in epi(\phi),$$
		where $\delta \eta + (1- \delta) \mu = \alpha. $
		
		This implies that $M$ is an $X$-convex set.
		\end{proof}
		\end{proposition}
			
		\begin{theorem}
		
		Let  $ \phi : M \rightarrow R $ be defined on a subset $M$ of $R^n$, also $M$ is chosen to be non-empty $X$-convex set . Let $I$ be a $X$-convex set in $R$ that contains $\phi(M)$. If $ \theta : I \rightarrow R $ is an increasing convex function, then $ \theta \circ \phi $ is an $X$-convex function on $M$.
		 \begin{proof}
		 Since $\phi$ is an $X$-convex function, then  $\forall$ $r,t \in M,$  $0 \leq \delta \leq 1,$  $\exists$ a vector valued function  $ g : M \rightarrow R^n $ such that
		 $$ \phi(\delta (r-t) + g(t)) \leq \delta \phi(r) + (1- \delta) \phi(t) , $$
		 and $$ \theta \circ(\phi(\delta (r-t) + g(t)) \leq  \theta(\delta \phi(r) + (1- \delta) \phi(t)),$$
		  $$ \theta \circ(\phi(\delta (r-t) + g(t)) \leq  \delta  \theta \circ \phi(r) + (1- \delta) \theta \circ \phi(t). $$
		 This implies that $\theta \circ \phi $ is an $X$-convex function on $M$.
		\end{proof}	  
		\end{theorem}
	\begin{theorem}
		For $M$, to be a non-empty $X$-convex subset of $R^n$, if that
		\\(a) $ \phi_1,~\phi_2 : M \rightarrow R $ be two $X$-convex functions with condition is same vector valued map $g$ associated with $X$-convexity of $\phi_1$ and $\phi_2$, then $ \phi_1 + \phi_2 $ also an $X$-convex function on $M$. 
		\\(b) $ \phi : M \rightarrow R $ is an $X$-convex function on $M,$ then for $\alpha \geq 0$, $\alpha \phi $ is also an $X$- convex function on $M$.\\
		(c) $\phi_i ,~~ i = 1, 2,.......,n $ are $X$-convex functions  with the condition that same vector valued map $g$ associated with $X$-convexity of all $\phi_i$, and $c_i \geq 0, $ then $\sum^{n}_{i=1} ~ {c_i\phi_i} $ is also an $X$-convex function on $M$.
		 \begin{proof}
 (a)  Let $\phi_1,\phi_2$ be $X$-convex function on $M$, then  $\forall$ $ r ,t \in M$ and  $0 \leq \delta \leq 1, $  $\exists$  vector valued maps $ g : M \rightarrow R^n $ with respect to $\phi_1, \phi_2 $ respectively, satisfying 
	
	\begin{equation}
	\phi_1(\delta (r-t) + g(t)) \leq \delta \phi_1(r) + (1- \delta) \phi_1(t) ,
	\end{equation}	
	and
 \begin{equation}
 \phi_2(\delta (r-t) + g(t)) \leq \delta \phi_2(r) + (1- \delta) \phi_2(t) ,
 \end{equation}  
 so, adding (4.2) and (4.3), we get
 	$$ \phi_1(\delta (r-t) + g(t)) + \phi_2(\delta (r-t) + g(t))  \leq \delta \phi_1(r) + (1- \delta) \phi_1(t) + \delta \phi_2(r) + (1- \delta) \phi_2(t), $$	
		$$( \phi_1 + \phi_2)(\delta (r-t) + g(t))  \leq \delta (\phi_1+\phi_2)(r) + (1- \delta) (\phi_1+\phi_2)(t). $$
		This implies that $\phi_1+\phi_2$ is $X$- convex function on $M$.
	\\	(b)  Suppose $\phi$ is $X$-convex function on $M$, then  $\forall$ $ r ,t \in M$ and  $0 \leq \delta \leq 1, $  $\exists$ a vector valued map $ g : M \rightarrow R^n, $ s.t., 
	 	$$ \phi(\delta (r-t) + g(t)) \leq \delta \phi(r) + (1- \delta) \phi(t),$$
	multiplying by $\alpha (\geq 0) $ in above equation, yields
	$$ \alpha (\phi(\delta (r-t) + g(t))) \leq \alpha( \delta \phi(r) + (1- \delta) \phi(t)),$$  
		$$ \alpha( \phi(\delta (r-t) + g(t))) \leq \delta(\alpha \phi(r)) + (1- \delta) (\alpha \phi(t)).$$
		This confirms $X$- convexity of $\alpha$$\phi$ on $M$.
		
		(c)  Suppose $\phi_i, i=1, 2,.....,n $ are $X$-convex function on $M$,  then  $\forall$ $ r ,t \in M$ and  $0 \leq \delta \leq 1, $  $\exists$  vector valued maps $ g : M \rightarrow R^n $ such that,  
		
			$$ \phi_i(\delta (r-t) + g(t)) \leq \delta \phi_i(r) + (1- \delta) \phi_i(t).$$
		It follows that $$ c_i \phi_i(\delta (r-t) + g(t)) \leq \delta c_i \phi_i(r) + (1- \delta) c_i \phi_i(t) $$ and 	$$(\sum^{n}_{i=1} ~ {c_i \phi_i} )(\delta (r-t) + g(t)) \leq \delta (\sum^{n}_{i=0} ~ {c_i \phi_i})(r) + (1- \delta) (\sum^{n}_{i=0} ~ {c_i \phi_i})(t).$$
	\end{proof}		
	\end{theorem}

\begin{theorem}

	$ \phi : M \rightarrow R $ be a function defined on a non-empty $X$-convex subset $M$ of $R^n$, then its lower level set is also an $X$-convex set.
	
	\begin{proof}
	
	Definition of $X$-convexity of the function $\phi$ says that $\forall$ $ r ,t \in M$ and  $0 \leq \delta \leq 1, $  $\exists$ a vector valued map $ g : M \rightarrow R^n,$ s.t., 	$$ \phi(\delta (r-t) + g(t)) \leq \delta \phi(r) + (1- \delta) \phi(t),$$
	
	and let $ r,t \in L(\phi,\eta),$ the lower level set. Then $$ \phi(r) \leq \eta ~\mbox{and}~ \phi(t) \leq \eta,$$ 
	from, 	$$ \phi(\delta (r-t) + g(t)) \leq \delta \eta + (1- \delta) \eta,$$
	$$ \phi(\delta (r-t) + g(t)) \leq\eta.	$$
	This implies that 	$$ (\delta (r-t) + g(t)) \in L(\phi,\eta).$$
	Therefore, the lower level set of $\phi$ is an $X$-convex set.	
\end{proof}
\end{theorem}
\begin{theorem}

Suppose $M$ be an $X$-convex subset of $\mathbb{R}^n$ with condition $\|\delta (s-r) + g(r)- r\| < \nu$ for each $ s, r \in M$, $\delta \in [0, 1]$, where $ g : M \rightarrow \mathbb{R}^n $ is a vector valued map associated with $X$-convexity and $\nu>0$. Then, every local minimum of a (strictly) $X$-convex function is a  global minimum (unique) of $ \phi$ over $M$.\\
 An $X$-convex set is also a collection of points on set $M$ where an $X$-convex function reaches its global minimum.  
\begin{proof} Suppose $\phi$ be an  $X$-convex function defined on a non-empty subset $M$ of $R^n$, attains its local minimum at $r \in M$. Then $\exists$ $\nu > 0 $, such that 
\begin{equation}
	\phi(r) \leq \phi(t), \text{ for all } t \in M \cap B _{\nu}(r)
	\end{equation}
Contrarily suppose $r \in M$, is not a global minimum of $\phi$ over $M$. Therefore, $\exists$ $ s (\neq r) \in M$,  such that, $ \phi(s) < \phi(r).$ Due to definition of $X$-convexity of $\phi$, we write 	$$ \phi(\delta (s-r) + g(r)) \leq \delta \phi(s) + (1- \delta) \phi(r),$$
$$\phi(\delta (s-r) + g(r)) < \delta \phi(r) + (1- \delta)\phi(r),$$
$$ \phi(\delta (s-r) + g(r)) <  \phi(r).$$
We get the contradiction of inequality (4.4), since $\|\delta (s-r) + g(r)- r\| < \nu$ for each $ s, r \in M$, $\delta \in [0, 1]$ then, $\delta (s-r) + g(r) \in M \cap B _{\nu}(r)$ for any $\delta$.

Now it's remain to show that $r$  is a unique minimum of $\phi$, if $\phi$ is strictly $X$-convex function. We suppose (on contrary), that $\exists$ $ r_{0}(\neq r)$ such that $r_{0} $ is also a global minimum of $\phi$ over the same set $M$, so we have $\phi(r)=\phi(r_0)$. Strictly $X$-convexity of $\phi$ for $0< \delta < 1$, gives

	$$ \phi(\delta (r_{0}-r) + g(r)) < \delta \phi(r_{0}) + (1- \delta) \phi(r),$$
	$$ \phi(\delta (r_{0}-r) + g(r)) < \delta \phi(r) + (1- \delta) \phi(r),$$
		$$ \phi(\delta (r_{0}-r) + g(r)) <  \phi(r).$$
	
		This is in direct conflict with the fact that $r$ is a global minimum of $\phi$ over $M$.
     	\\ 	Suppose we have a set $ A = \{r \in M : \phi(r) \leq \phi(t) $ for all $ t \in M\}$ at which $\phi$ attains its global minimum points. If $ r_{1}, r_{2} \in A$, then $\phi(r_{1}) \leq \phi(t)$ and $\phi(r_{2}) \leq \phi(t) $ for each $ t \in M. $  The definition of $X$-convexity of function $\phi$, leads to 
     	
     		\begin{equation}
     	\phi(\delta (r_{2}-r_{1}) + g(r_{1})) \leq \delta \phi(r_{2}) + (1- \delta) \phi(r_{1}),
     		\end{equation}
     	For $0\leq \delta \leq 1$ equation (4.5), in preview of global minimum at $A$ implies
     	
     		$$ \phi(\delta (r_{2}-r_{1}) + g(r_{1})) \leq \delta \phi(t) + (1- \delta) \phi(t),$$
     		$$ \phi(\delta (r_{2}-r_{1}) + g(r_{1})) \leq \phi(t),$$
     		which implies that $\delta(r_{2}-r_{1}) + g(r_{1}) \in A.$ Hence, $A$ is $X$-convex. 
	\end{proof}
\end{theorem}

\begin{theorem}
	 For an $X$-convex non-empty subset $M$ of $R^n$, a function $ \phi : M \rightarrow R $ is quasi-$X$-convex iff $ L(\phi, \eta), $ the lower level sets are $X$-convex $\forall$ $\eta \in R.$ 
	\begin{proof}
	
	 Suppose $\phi$ be a quasi-$X$-convex function and for $\eta \in R, $ let $ r,t \in L(\phi,\eta). $ Then $ \phi(r) \leq \eta$ and $ \phi(t) \leq \eta. $ Due to quasi-$X$-convexity of $\phi$, we can write for $0 \leq \delta \leq 1$
	 
	$$ \phi(\delta (r-t) + g(t)) \leq \mbox{max}\{\phi(r) , \phi(t)\}, $$ 
		$$ \phi(\delta (r-t) + g(t)) \leq \eta, $$
		i.e.,  $\delta (r-t) + g(t) \in L(\phi,\eta)$ $\implies$ $ L(\phi,\eta) $ is an $X$-convex set.\\
	For the converse part, suppose for any $r, t \in M $ and $ {\eta}' =\mbox{max}\{\phi(r), \phi(t)\},$ then $r, t \in L(\phi, {\eta}'), $ and due to $X$-convexity of $ L(\phi, {\eta}'), $ we have  $ \delta (r-t) + g(t) \in L(\phi, {\eta}') ~\forall~ 0 \leq \delta \leq 1.$ Thus, we have	$$ \phi(\delta (r-t) + g(t)) \leq {\eta}'= \mbox{max}\{\phi(r) , \phi(t)\}. $$ 
	\end{proof}
\end{theorem}

\begin{theorem}
	Suppose $M$ be an $X$-convex subset of $\mathbb{R}^n$ with condition $\|\delta (s-r) + g(r)- r\| < \nu$ for any $ s, r \in M$, $\delta \in [0, 1]$, where $ g : M \rightarrow \mathbb{R}^n $ is a vector valued map associated with $X$-convexity and $\nu>0$.
If  $ \phi : M \rightarrow R $ is a quasi-$X$-convex function, then any strict local minimum of $\phi$ over $M$ is also a strict global minimum.
Furthermore, the set of points where $M$ reaches its global minimum is an $X$-convex set.

\begin{proof}
   Suppose $r \in M $ attains local minimum of $\phi$  at a point $r$ of $M$. Then $\exists$ $\nu > 0 $ such that 
\begin{equation}
	\phi(r) < \phi(t),~ \forall t  \in M \cap B_{\nu}(r).
\end{equation}
Contrarily we suppose that $r$ is not a strict global minimum over $M$. Then $\exists$ $ r_{0} (\neq r) \in M $ such that $ \phi(r_{0}) \leq \phi(r).$ Due to quasi-$X$-convexity of $\phi$, we have for $ \delta \in\ [0,1],$ $\exists$ a vector valued function $g$ such that 	$$ \phi(\delta (r_{0}-r) + g(r)) \leq \phi(r).$$ This contradicts (4.6), since $\|\delta (r_{0}-r) + g(r)- r\| < \nu$ for each $ r_{0}, r \in M$, $\delta \in [0, 1]$ then, $\delta (r_{0}-r) + g(r) \in M \cap B _{\nu}(r)$ for any $\delta$.
\\ Now we suppose $\eta=\mbox{min}\{\phi(t) : t \in M\}$ be the global minimum value attained by $\phi$, $ L(\phi, \eta)$ is the collection of places where its global minimum on $M$ occurs, say $\eta$, that is,  the set	of global minimum points over $M$, and due to quasi-$X$-convexity of $\phi$, it is an $X$-convex set.	
	\end{proof}
\end{theorem}

\begin{theorem}
Any strictly quasi-$X$-convex function  $ \phi : M \rightarrow R $  defined over a non-empty $X$-convex subset $M$ of $R^n$ attains its minimum  at the most one point of $M$.
\begin{proof}
 Suppose on contrary, that $\phi$ attains its minimum at two distinct points $u_1$ and $u_2 \in M.$ Then
 \begin{equation}
 	\phi(u_1) \leq \phi(t)~ \mbox{and}~ \phi(u_2) \leq \phi(t), ~\forall t \in M.
\end{equation}

For $t=u_2$ in the first inequality and $t=u_1$ in the second inequality, we get $\phi(u_1)=\phi(u_2)$. By strict quasi-$X$-convexity of $\phi$, we have $\forall$ $0 < \delta <1, $ $\exists$ a vector valued function $g$ such that $$  \phi(\delta (u_1-u_2) + g(u_2)) < \phi(u_1),$$ which contradicts (4.7).
	\end{proof}
\end{theorem}
\begin{theorem}
	Consider a quasi-$X$-convex function  $ \phi : M \rightarrow R $ defined on a non-empty $X$-convex subset $M$ of $R^n$. Suppose $I$ be an $X$-convex set in $R$ that contains $\phi(M)$. If $\theta : I \rightarrow R$ is non-decreasing convex function, then $ \theta \circ \phi $ is an $X$-convex function on $M$.
	\begin{proof}
 As $\phi$ is quasi-$X$-convex function on $M$, then for each $r, t \in M,$ $0\leq\delta \leq 1$ and $\exists$ a vector valued function  $ g : M \rightarrow R^n $ such that
$$ \phi(\delta (r-t) + g(t)) \leq \mbox{max}\{\phi(r), \phi(t)\},$$
and $$\theta \circ(\phi(\delta (r-t) + g(t)) \leq  \theta (\mbox{max}\{\phi(r), \phi(t)\}),$$
$$\theta \circ(\phi(\delta (r-t) + g(t)) \leq  \mbox{max}\{ \theta \circ \phi(r),  \theta \circ \phi(t)\}.$$
This implies that $\theta\circ \phi $ is a quasi-$X$-convex function on $M$.
\end{proof}
\end{theorem}
\begin{theorem}
		Suppose $M$ be an $X$-convex subset of $\mathbb{R}^n$ with condition $\|\delta (s-r) + g(r)- r\| < \nu$ for each $ s, r \in M$, $\delta \in [0, 1]$, where $ g : M \rightarrow \mathbb{R}^n $ is a vector valued map associated with $X$-convexity and $\nu>0$.
On a semi-strictly quasi-$X$-convex function $\phi$,  every local minimum is a global minimum of $\phi$ over $M$.
\begin{proof}
 
  As the proof of theorem (4.5), we can write similarly the proof of theorem (4.10).
  
\end{proof}
\end{theorem}
\section{Application of Quasi-$X$-Convex Type Functions}
In this part, we\textsc{\char13}ll see that how quasi-X-convex, semi-strictly quasi-X-convex, and strictly quasi-X-convex functions can be used to characterize the solution for distinct optimization problems. The motivation behind this work is the study of prequasi-invex functions by X.M. Yang et al, \cite{Yang} in which they have dealt with a minimization problem. Here we also follow the definitions and results they devolped in their work.

Consider the following multi-objective optimization problem.
\begin{equation}
	\min \phi(r)=(\phi_1(r),\phi_2(r),...,\phi_p(r))^T, ~~~~s.t.~ r \in M,
\end{equation}
where $ \phi : M \rightarrow \mathbb{R}^p$ is a vector-valued function and $ M$ be an $X$-convex subset of $ \mathbb{R}^n$  with respect to $ g : M \rightarrow \mathbb{R}^n$.

Let  $A$ and $A'$ and we define
$$ A= \{ \mu \in R^p~ |~ \mu= (\mu_1, \mu_2,..., \mu_p), \mu_i \geq0, i=1,2,...,p\},$$
$$ A'= \{ \mu \in R^p~ |~ \mu= (\mu_1, \mu_2,..., \mu_p), \mu_i >0, i=1,2,...,p\}.$$
\begin{definition}
	 
	Let $r \in M$ and it is known as a global efficient solution of problem (5.1) if there $\nexists$ any point $t \in M$ such that $$ \phi(t) \in \phi(r)-A \backslash\{0\}.$$
	
	Let $r \in M$ and it is known as a local efficient solution of problem (5.1) if there is a neighborhood $N(r)$ of $r$ such that there $\nexists$ any point $ t \in M \cap N(r)$ such that $$ \phi(t) \in \phi(r)-A \backslash\{0\}.$$ 
\end{definition}
\begin{definition}
 
Let $r \in M$ and it is known as a global weakly efficient solution of problem (5.1) if there$\nexists$ any point $t \in M$ such that $$ \phi(t) \in \phi(r)-A' .$$	

Let $r \in M$ and it is known as a local weakly efficient solution of problem (5.1) if there is a neighbourhood $N(r)$ of $r$ such that  $\nexists$ any point $ t \in M \cap N(r)$ such that $$ \phi(t) \in \phi(r)-A' .$$	
\end{definition}
\begin{theorem}
	Suppose $\phi_i(r),$ $ i=1,2,...,p$, be quasi-$X$-convex and semi-strictly quasi-$X$-convex function with respect to the same vector-valued function $g$. Any local efficient solution to problem (5.1) is thus a global efficient solution to a problem (5.1).
\end{theorem}
	\begin{proof}
		Contrarily we suppose that $\exists$ an $r \in M$ is a local efficient solution of problem (5.1), but is not a global efficient solution of problem (5.1). Then, $\exists$  $s \in M$ such that 
		\begin{equation}
			 \phi_i(s)\leq \phi_i (r), ~~ 1 \leq i \leq p
		\end{equation}
	and for some $j$,
	\begin{equation}
		 \phi_j(s)< \phi_j(r).
	\end{equation}
From the quasi-$X$-convexity of $\phi_i(r)$; and using (5.2), we get
\begin{equation}
	\phi_i(\delta(r-t)+g(t))\leq \phi_i(r),  ~~ 1 \leq i \leq p, ~~\forall \delta \in[0, 1],
\end{equation}
and from the semi-strictly quasi-$X$-convexity of $\phi_j$ and using (5.3), we get
\begin{equation}
	\phi_j(\delta(r-t)+g(t))< \phi_j(r),  ~~  \forall \delta \in[0, 1];
\end{equation}
(5.4) and (5.5) demonstrate that r is not a local efficient  solution to the problem (5.1), a contradiction.
	\end{proof}
\begin{theorem}
	Suppose $\phi_1(r), \phi_2(r),..., \phi_p(r)$ be quasi-$X$-convex functions with respect to the same vector-valued function $g$ and, for some $k$, suppose $\phi_k(r)$ be a strictly quasi-$X$-convex function with respect to the same vector function $g$. Let us consider that there exists a $\mu=(\mu_1,\mu_2,...,\mu_p)\geq 0,$ with $\mu_k>0$, such that $ r\in M$ is a local  solution of $\min \mu^T \phi(r)$, such that $ r\in M$. Then, $r$ is also a global efficient solution of problem (5.1).
\end{theorem}
\begin{proof}
	Contrarily we suppose that $r$ is not a global efficient solution of problem (5.1); i.e., $ \exists$ some $t \in M$, $\phi(t)\neq\phi(r),$ such that $$ \phi_i(t)\leq \phi_i(r),~~ 1\leq i\leq p.$$ Then, for any $ \delta \in (0, 1)$, from the quasi-$X$-convexity of $\phi_i$, we get $$ \phi_i(\delta(r-t)+g(t))\leq \phi_i(r),~~ 1\leq i\leq p$$ and from the strict quasi-$X$-convexity of some $\phi_k$,  $$ \phi_k(\delta(r-t)+g(t))\leq \phi_k(r).$$ Hence, by  $\mu=(\mu_1,\mu_2,...,\mu_p)\geq 0,$ with $\mu_k>0$, $$ \sum_{i=1}^{p} \mu_i\phi_i(\delta(r-t)+g(t))<\sum_{i=1}^{p}\mu_i\phi_i(r),~~ \delta \in (0, 1).$$
	That is, $$  \mu^T\phi(\delta(r-t)+g(t))<\mu^T\phi(r),~~ \delta \in (0, 1).$$ We find a contradiction.
\end{proof}
\begin{theorem}
		Suppose $\phi_1(r), \phi_2(r),..., \phi_p(r)$ be quasi-$X$-convex functions with respect to the same vector-valued function $g$ and, for some $k$, suppose $\phi_k(r)$ be a strictly quasi-$X$-convex function with respect to the same vector function $g$.  Any local efficient solution to problem (5.1) is then also a global efficient solution to problem (5.1).
\end{theorem}
On similiar lines the following two theorems can also be proved.
\begin{theorem}
	Suppose $\phi_1(r), \phi_2(r),..., \phi_p(r)$ be semi-strictly quasi-$X$-convex with respect to the same vector-valued function $g$. Then, every local weakly  efficient solution of problem (5.1) is a global weakly efficient solution of problem (5.1).	
\end{theorem}
\begin{theorem}
	Suppose $\phi_1(r), \phi_2(r),..., \phi_p(r)$ be semi-strictly quasi-$X$-convex functions with respect to the same vector-valued function $g$.  suppose there exists  $\mu=(\mu_1,\mu_2,...,\mu_p)\geq 0,$ with $\mu_k>0$, such that $ r\in M$ is a local  solution of $\min \mu^T \phi(r)$, such that $ r\in M$. Then, $r$ is also a global weakly efficient solution of problem (5.1).	
\end{theorem}
When $p = 1$, the multi-objective mathematical programming problem (5.1) becomes a single-objective mathematical programming problem,
\begin{equation}
	\min \phi(r) ~~~~s.t.~ r \in M,
\end{equation}
where $ \phi : M \rightarrow \mathbb{R}$ is a real-valued function and $ M$ is an $X$-convex subset of $\mathbb{R}^n$  with respect to $ g : M \rightarrow \mathbb{R}^n$.
\begin{theorem}
	Let $\phi$ be strictly quasi-$X$-convex with respect to the vector valued function $ g : M \rightarrow \mathbb{R}^n$. Then, the solution of problem (5.6) is unique.
\end{theorem}
\begin{proof}
	Let $r$ be a solution of problem (5.1). Contrarily we suppose that $r$ is not the unique solution. Then, $\exists$ $t \in M$ such that $r\neq t$ and $ \phi(r)=\phi(t)$. Since $\phi$ is a strictly quasi-$X$-convex function, we have $$ \phi(\delta(r-t)+g(t))\leq \phi(r), ~~\forall ~\delta \in (0, 1)$$ which implies that $r$ is not a solution of problem (5.6), a contradiction.
	
\end{proof}
\begin{theorem}
	Let $\phi$ be quasi-$X$-convex with respect to the vector-valued function $g$. Then the solution set of problem (5.6) is $X$-convex with respect to $g$.
\end{theorem}
\begin{proof}
	Let
	$$ \beta=\inf_{r \in M} \phi(r) ~~ \mbox{and}~~ G= \{r\in M : \phi(r)=\beta\}.$$
	Now for any $r_1, r_2 \in G$, by definition of quasi-$X$-convexity of $\phi$ on $M$ with respect to $g$, we have $$ \phi(\delta(r_1-r_2)+g(r_2))\leq \max \{\phi(r_1), \phi (r_2)\}=\beta, \delta\in [0, 1],$$
	$\implies$ $\delta(r_1-r_2)+g(r_2) \in G,~~ 0\leq \delta \leq 1. $\\
	Therefore, the solution set $G$ of problem (5.6) is $X$-convex with respect to $g$.
\end{proof}

\end{document}